\documentclass[11pt,reqno]{amsart}
\usepackage{graphicx}
\usepackage{verbatim}
\usepackage{textcomp}
\usepackage{amssymb}
\usepackage{cite}
\usepackage{amsmath}
\usepackage{latexsym}
\usepackage{amscd}
\usepackage{amsthm}
\usepackage{mathrsfs}
\usepackage{xypic}
\usepackage{bm}
\usepackage{url}

\newtheorem{thm}{Theorem}
\newtheorem{corr}[thm]{Corollary}
\newtheorem{lem}[thm]{Lemma}
\newtheorem{prop}[thm]{Proposition}

\newtheorem{claim}[thm]{Claim}
\theoremstyle{definition}

\newtheorem{rem}[thm]{Remark}
\def\R{\mathbb R}

\def\SS{\mathbb S}

\def\f{\frac}

\def\ra{\rightarrow}
\def\pt{\partial}

\begin{document}
\title[Note on Brendle-Eichmair's paper]{Note on Brendle-Eichmair's paper ``Isoperimetric and Weingarten surfaces in the Schwarchild manifold''}
\author{Haizhong Li}
\address{Department of mathematical sciences, and Mathematical Sciences
Center, Tsinghua University, 100084, Beijing, P. R. China}
\email{hli@math.tsinghua.edu.cn}
\author{Yong Wei}
\address{Department of mathematical sciences, Tsinghua University, 100084, Beijing, P. R. China}
\email{wei-y09@mails.tsinghua.edu.cn}
\author{Changwei Xiong}
\address{Department of mathematical sciences, Tsinghua University, 100084, Beijing, P. R. China}
\email{xiongcw10@mails.tsinghua.edu.cn}
\date{\today}
\thanks{The research of the authors was supported by NSFC No. 11271214.}

\maketitle

\begin{abstract}
In this short note, we show that the assumption ``convex'' in Theorem $7$ of Brendle-Eichmair's paper \cite{BE} is unnecessary.
\end{abstract}

\section{Introduction}
For $n\geq 3$, let $\lambda:[0,\bar{r})\ra \R$ be a smooth positive function which satisfies the following conditions (see \cite{BE}):
\begin{itemize}
  \item[(H1)] $\lambda'(0)=0$ and $\lambda''(0)>0$.
  \item[(H2)] $\lambda'(r)>0$ for all $r\in(0,\bar{r})$.
  \item[(H3)] The function $$2\f{\lambda''(r)}{\lambda(r)}-(n-2)\f{1-\lambda'(r)^2}{\lambda(r)^2}$$ is non-decreasing for $r\in(0,\bar{r})$.
  \item[(H4)] $\f{\lambda''(r)}{\lambda(r)}+\f{1-\lambda'(r)^2}{\lambda(r)^2}>0$ for all $r\in (0,\bar{r})$.
\end{itemize}

Now we consider the manifold $M=\SS^{n-1}\times [0,\bar{r})$ equipped with a Riemannian metric $\bar{g}=dr\otimes dr+\lambda(r)^2 g_{\SS^{n-1}}$.
Let $\Sigma$ be a closed embedded star-shaped hypersurface in $(M,\bar{g})$, where star-shaped means that the unit outward normal $\nu$ satisfies $\langle \pt_r,\nu\rangle\geq 0$. Denote by $\sigma_p$ the $p$-th elementary symmetric polynomial of the principal curvatures. In fact, for this manifold $(M,\bar{g})$ Brendle and Eichmair proved the following theorem
\begin{thm}[Theorem 7 of \cite{BE}]\label{thm-1}
Let $\Sigma$ be a closed embedded hypersurface in the manifold $(M,\bar{g})$ that is star-shaped and convex. If $\sigma_p$ is constant, then $\Sigma$ is a slice $\SS^{n-1}\times \{r\}$ for some $r\in (0,\bar{r})$.
\end{thm}

In this note, we show that the assumption `` convex'' in Theorem \ref{thm-1} is unnecessary. That is we have
\begin{thm}\label{thm-2}
Let $\Sigma$ be a closed, embedded and star-shaped hypersurface in the manifold $(M,\bar{g})$. If $\sigma_p$ is constant, then $\Sigma$ is a slice $\SS^{n-1}\times \{r\}$ for some $r\in (0,\bar{r})$.
\end{thm}

Note that the conditions (H1)-(H4) are all satisfied on the deSitter-Schwarzschild manifolds (see \cite{BE}). So we have the following Corollary
\begin{corr}
Let $\Sigma$ be a closed, embedded and star-shaped hypersurface in the deSitter-Schwarzschild manifold $(M,\bar{g})$. If $\sigma_p$ is constant, then $\Sigma$ is a slice $\SS^{n-1}\times \{r\}$ for some $r\in (0,\bar{r})$.
\end{corr}

\section{Proof of Theorem \ref{thm-2}}

In this section, by observing the existence of an elliptic point on $\Sigma$ and some basic facts about the function $\sigma_p$, we can remove the assumption ``convex''.

Let $X=\lambda(r)\pt_r$. It is easy to see that $X$ is a conformal vector field satisfying $\bar{\nabla}X=\lambda'\bar{g}$. Following the argument as Lemma 5.3 in \cite{AC}, we have
\begin{lem}\label{lem-1}
Let $\psi:\Sigma\ra (M,\bar{g})$ be a closed  hypersurface. Then there exists an elliptic point $x$ on $\Sigma$, i.e., all the principal curvatures are positive at $x$.
\end{lem}
\proof
Let $h=\pi_I\circ\psi:\Sigma\ra I$ be the height function on $\Sigma$, where $\pi_I$ is the projection $\pi_I(r,\theta)=r$. At any point $x\in\Sigma$, we have
\begin{equation}
    \nabla h=(\bar{\nabla}r)^{\top}=(\pt_r)^{\top}.
\end{equation}
Let $\{e_1,\cdots,e_{n-1}\}$ be a local orthonormal frame on $\Sigma$, and assume that the second fundamental form $h_{ij}=\langle\bar{\nabla}_{e_i}\nu,e_j\rangle$ is diagonal with eigenvalues $\kappa_1,\cdots,\kappa_{n-1}$. Then
\begin{eqnarray}
\nabla_{e_i}\nabla h &=& \nabla_{e_i}(\frac 1{\lambda(h)}\lambda(h)\pt_r^{\top})\nonumber \\
&=&-\frac{\lambda'}{\lambda}(\nabla_{e_i} h)\pt_r^{\top}+\frac 1{\lambda}\nabla_{e_i}(\lambda\pt_r^{\top}).\label{eq-1}
\end{eqnarray}
Note that $X=\lambda\pt_r$ is a conformal vector field, we have
\begin{align}
    \nabla_{e_i}(\lambda\pt_r^{\top})=&\nabla_{e_i}(\lambda\pt_r-\langle\lambda\pt_r,\nu\rangle\nu)\nonumber\\
    =&(\bar{\nabla}_{e_i}(\lambda\pt_r-\langle\lambda\pt_r,\nu\rangle\nu))^{\top}\nonumber\\
    =&\lambda'e_i-\langle\lambda\pt_r,\nu\rangle\kappa_ie_i.\label{eq-2}
\end{align}
Substituting \eqref{eq-2} into \eqref{eq-1} gives that
\begin{equation}\label{eq-3}
   \nabla_{e_i}\nabla h=-\frac{\lambda'}{\lambda}(\nabla_{e_i} h)\pt_r^{\top}+\frac{1}{\lambda}(\lambda'-\langle\lambda\pt_r,\nu\rangle\kappa_i)e_i.
\end{equation}
Now we consider the maximum point $x$ of $h$. We have $\nabla h=0, \nu=\pt_r$ and $\nabla^2h\leq 0$ at $x$. Then from \eqref{eq-3}, we get
\begin{equation*}
    \kappa_i\geq \frac{\lambda'}{\lambda}>0,\qquad i=1,\cdots,n-1,
\end{equation*}
i.e., $x$ is an elliptic point of $\Sigma$.
\endproof

\begin{rem}
If we assume that the closed embedded hypersurface $\Sigma$ in $M$ satisfies $\langle\pt_r,\nu\rangle>0$, then $\Sigma$ can be parametrized by a graph on $\SS^{n-1}$ (see \cite{BHW}):
\begin{equation*}
    \Sigma=\{(r(\theta),\theta):\theta\in \SS^{n-1}\}.
\end{equation*}
Define a function $\varphi:\SS^{n-1}\ra\R$ by $\varphi(\theta)=\Phi(r(\theta))$, where $\Phi(r)$ is a positive function satisfying $\Phi'=1/{\lambda}$. Let $\varphi_i,\varphi_{ij}$ be covariant derivatives of $\varphi$ with respect to $g_{\SS^{n-1}}$. Define $v=\sqrt{1+|\nabla\varphi|^2_{g_{\SS^{n-1}}}}$. Then the same calculation as in Proposition 5 in \cite{BHW} gives that the second fundamental form of $\Sigma$ has the expression
\begin{equation*}
    h_{ij}=\frac{\lambda'}{v\lambda}g_{ij}-\frac{\lambda}{v}\varphi_{ij},
\end{equation*}
where $g_{ij}$ is the induced metric on $\Sigma$ from $(M,\bar{g})$. At the maximum point $x$ of $\varphi$, we have $\varphi_i=0,\varphi_{ij}\leq 0$. Then we have $h_{ij}\geq \frac{\lambda'}{\lambda}g_{ij}$, i,e, $x$ is an elliptic point of $\Sigma$. Note that the maximum point $x$ of $\varphi$ is also a maximum point of $r$.
\end{rem}

Recall that for $1\leq k\leq n-1$ the convex cone $\Gamma^+_k\subset\R^{n-1}$ is defined by
\begin{equation*}
    \Gamma^+_k=\{\vec{\kappa}\in\R^{n-1}|\sigma_{j}(\vec{\kappa})>0\textrm{ for }j=1,\cdots,k\},
\end{equation*}
or equivalently
\begin{equation*}
    \Gamma^+_k=\textrm{component of }\{\sigma_k>0\}\textrm{ containing the positive cone}.
\end{equation*}
It is clearly that $\Gamma_k^+$ is a cone with vertex at the origin and $\Gamma_k^+\subset\Gamma_j^+$ for $j\leq k$. We write $\sigma_0=1$, $\sigma_k=0$ for $k>n-1$, and denote $\sigma_{k;i}(\vec{\kappa})=\sigma_k(\vec{\kappa})|_{\kappa_i=0}$, i.e. $\sigma_{k;i}(\vec{\kappa})$ is the $k$-the elementary symmetric polynomial of $(\kappa_1,\cdots,\kappa_{i-1},\kappa_{i+1},\\\cdots,\kappa_{n-1})$. Then we have the following classical result (see, e.g, \cite[Lemma 2.3]{STW},\cite{CNS,LT,HLMG,Ga}).
\begin{lem}\label{lem-3}
If $\vec{\kappa}\in\Gamma_k^+$, then $\sigma_{k-1;i}(\vec{\kappa})>0$ for each $1\leq i\leq n-1$ and
\begin{equation}\label{inequ}
    \sigma_{j-1}\geq \frac j{n-j}\begin{pmatrix}n-1\\j\end{pmatrix}^{1/j}\sigma_j^{(j-1)/j},\textrm{ for }1\leq j\leq k.
\end{equation}
\end{lem}

\vskip 2mm
The following Lemma shows that on connected closed hypersurface in $(M,\bar{g})$, the positiveness of $\sigma_p$ implies that the principal curvatures $\vec{\kappa}\in\Gamma_p^+$.
\begin{lem}\label{lem-2}
Let $\Sigma$ be a connected, closed hypersurface in $(M,\bar{g})$. If $\sigma_p>0$ on $\Sigma$, then we have $\sigma_j>0$ on $\Sigma$ for each $1\leq j\leq p-1$.
\end{lem}
\proof
We believe that the proof of this Lemma can be found in literature, for example, see the proof of Proposition 3.2 in \cite{BC}. For convenience of the readers, we include the proof here. Lemma \ref{lem-1} implies that there exists an elliptic point $x$ on $\Sigma$. By continuity there exists an open neighborhood $\mathcal{U}$ around $x$ such that the principal curvatures are positive in $\mathcal{U}$. Hence $\sigma_k$ are positive in $\mathcal{U}$ for each $1\leq k\leq n-1$. Denote by $\mathcal{G}_j$ the connected component of the set $\{x\in \Sigma:\sigma_{j}|_x>0\}$ containing $\mathcal{U}$.

\begin{claim}\label{claim} For each $j$, we have $\mathcal{G}_{j+1}\subset \mathcal{G}_j$.
\end{claim}
\begin{proof}[Proof of the Claim]
For each $k$, define the open set
\begin{equation*}
    \mathcal{V}_k=\bigcap_{j=1}^k\mathcal{G}_j.
\end{equation*}
It suffices to show that $\mathcal{V}_k=\mathcal{G}_k$. Since $\sigma_{j}>0$ in $\mathcal{V}_k$ for $1\leq j\leq k$, Lemma \ref{lem-3} implies that at each point of this open set $\mathcal{V}_k$ the inequalities \eqref{inequ} hold. By continuity \eqref{inequ} also hold at the boundary of $\mathcal{V}_k$. If a point $y$ of the boundary of $\mathcal{V}_k$ belongs to $\mathcal{G}_k$, then \eqref{inequ} implies $y\in \mathcal{G}_j$ for each $j\leq k$ and therefore belongs to $\mathcal{V}_k$. This shows that the boundary of $\mathcal{V}_k$ is contained in the boundary of $\mathcal{G}_k$. Since by definition $\mathcal{V}_k\subset\mathcal{G}_k$ and they are both open sets, $\mathcal{G}_k$ is connected, we have $\mathcal{V}_k=\mathcal{G}_k$. This completes the proof of the Claim.
\end{proof}

Now we continue the proof of Lemma \ref{lem-2}. We will show that $\mathcal{G}_{p-1}$ is closed. Pick a point $y$ at the boundary of $\mathcal{G}_{p-1}$. By continuity $\sigma_{p-1}\geq 0$ at $y$. Then Claim \ref{claim} implies that $\sigma_{j}\geq 0$ at $y$ for each $1\leq j\leq p-1$.
If $\sigma_{p-1}=0$ at $y$, by hypothesis $\sigma_p>0$ and using Lemma \ref{lem-3}, we have
\begin{equation*}
    0=\sigma_{p-1}\geq \frac p{n-p}\begin{pmatrix} n-1\\p\end{pmatrix}^{1/{p}}\sigma_{p}^{(p-1)/p}>0,
\end{equation*}
which is a contradiction. This implies $\sigma_{p-1}\neq 0$ at $y$, and $y$ belongs to the interior of $\mathcal{G}_{p-1}$. Therefore $\mathcal{G}_{p-1}$ is closed. Since it is also open, and then $\mathcal{G}_{p-1}=\Sigma$ by the connectedness of $\Sigma$. Then Claim \ref{claim} shows that $\mathcal{G}_j=\Sigma$ for each $1\leq j\leq p-1$, this implies $\sigma_{j}>0$ for $1\leq j\leq p-1$ on $\Sigma$ and completes the proof of Lemma \ref{lem-2}.
\endproof

Now we can prove Theorem \ref{thm-2}. As in \cite{BE}, it suffices to prove the Heintze-Karcher-type inequality and Minkowski-type inequality.

If $\sigma_p$ is a constant on $\Sigma$, then Lemma \ref{lem-1} implies $\sigma_p=const>0$.
Denote by $\vec{\kappa}=(\kappa_1,\cdots,\kappa_{n-1})$ the principal curvatures of $\Sigma$. Then Lemma \ref{lem-2} implies
$\vec{\kappa}\in\Gamma^+_p$ on $\Sigma$. Thus $\vec{\kappa}\in \Gamma^+_1$ and $\Sigma$ is mean convex.
So the Heintze-Karcher-type inequality
\begin{align}\label{eq-HK}
(n-1)\int_{\Sigma}\frac {\lambda'}{H}\geq& \int_{\Sigma}\langle X,\nu\rangle
\end{align}
can be obtained as in \cite{B}.

On the other hand, we can prove
\begin{prop}[Minkowski-type inequality]
For $1\leq p\leq n-1$, suppose that $\Sigma$ is star-shaped and $\sigma_p>0$. Then
 \begin{align}\label{eq-M}
 p\int_\Sigma\langle X,\nu\rangle \sigma_p \geq (n-p)\int_\Sigma \lambda' \sigma_{p-1}
 \end{align}
\end{prop}
\begin{proof}
Let $\xi=X-\langle X,\nu\rangle \nu$ and $T_{ij}^{(p)}=\frac{\pt\sigma_p}{\pt h_{ij}}$. Then $$\nabla_i\xi_j=\bar{\nabla}_iX_j-\langle X,\nu\rangle h_{ij}=\lambda'\bar{g}_{ij}-\langle X,\nu\rangle h_{ij}$$ Therefore
\begin{eqnarray}
\sum_{i,j=1}^{n-1} \nabla_i(\xi_j T_{ij}^{(p)})&=&\lambda' \sum^{n-1}_{i=1} T_{ii}^{(p)}-\sum^{n-1}_{i,j=1}T_{ij}^{(p)}\langle X,\nu\rangle h_{ij}+\sum^{n-1}_{i,j=1}\xi_j\nabla_i T^{(p)}_{ij}\nonumber \\
&=&\lambda'(n-p)\sigma_{p-1}-p\sigma_p\langle X,\nu\rangle+\sum^{n-1}_{i,j=1}\xi_j\nabla_i T^{(p)}_{ij}\label{eq-4}
\end{eqnarray}
Next as the proof of Proposition $8$ in \cite{BE}, we can get
$$\sum^{n-1}_{i,j=1}\xi_j\nabla_i T^{(p)}_{ij}=-\frac{n-p}{n-2}\sum^{n-1}_{j=1}\sigma_{p-2;j}(\vec{\kappa})\xi_j Ric(e_j,\nu)$$
By direct calculation, we have
\begin{align*}
Ric(e_j,\nu)=&-(n-2)\left(\f{\lambda''(r)}{\lambda(r)}+\f{1-\lambda'(r)^2}{\lambda(r)^2}\right)\frac{\xi_j}{\lambda}\langle\pt_r,\nu\rangle.
\end{align*}
Thus, using the assumption
``star-shaped'' $\langle \pt_r,\nu\rangle\geq 0$ and the condition (H4), we have $\xi_j Ric(e_j,\nu)\leq 0$ for each $1\leq j\leq n-1$. On the other hand, from Lemma \ref{lem-2} and Lemma \ref{lem-3}, $\vec{\kappa}\in\Gamma^+_{p-1}$ on $\Sigma$ and  $\sigma_{p-2;j}(\vec{\kappa})>0$ for each $1\leq j\leq n-1$. Therefore we have
\begin{equation}\label{eq-5}
    \sum^{n-1}_{i,j=1}\xi_j\nabla_i T^{(p)}_{ij}\geq 0.
\end{equation}
Putting \eqref{eq-5} into \eqref{eq-4} and integrating on $\Sigma$, we get the Proposition 8.
\end{proof}

Once obtaining the Heintze-Karcher-type inequality \eqref{eq-HK} and the Minkowski-type inequality \eqref{eq-M}, we can go through the remaining proof as in \cite{BE}, which completes the proof of Theorem \ref{thm-2}.

\appendix
\section{Further remark}
Finally we give a remark about the generalization of Theorem \ref{thm-2}. For $n\geq 3$, let $(N,g_N)$ be a compact Einstein manifold of dimension $n-1$ satisfying $Ric_N=(n-2)Bg_N$ for some constant $B$. Moreover, let $\lambda:[0,\bar{r})\ra \R$ be a smooth positive function which satisfies the following conditions:
\begin{itemize}
  \item[$\textrm{(H1)}^\prime$] $\lambda'(0)=0$ and $\lambda''(0)>0$.
  \item[$\textrm{(H2)}^\prime$] $\lambda'(r)>0$ for all $r\in(0,\bar{r})$.
  \item[$\textrm{(H3)}^\prime$] The function $$2\f{\lambda''(r)}{\lambda(r)}-(n-2)\f{B-\lambda'(r)^2}{\lambda(r)^2}$$ is non-decreasing for $r\in(0,\bar{r})$.
  \item[$\textrm{(H4)}^\prime$] $\f{\lambda''(r)}{\lambda(r)}+\f{B-\lambda'(r)^2}{\lambda(r)^2}>0$ for all $r\in (0,\bar{r})$.
\end{itemize}
Let  manifold $M=N\times [0,\bar{r})$   with a Riemannian metric $\bar{g}=dr\otimes dr+\lambda(r)^2 g_N$. By use of the similar arguments as proof of Theorem 2, we can obtain the following  generalization of Theorem 2
\begin{thm}
Let $\Sigma$ be a closed, embedded and star-shaped hypersurface in the manifold $(M,\bar{g})$. If $\sigma_p$ is constant, then $\Sigma$ is a slice $N\times \{r\}$ for some $r\in (0,\bar{r})$.
\end{thm}


\bibliographystyle{Plain}

\end{document}